\documentclass{amsart}
\usepackage{amsmath}
\usepackage{amssymb}
\usepackage{mathtools}
\usepackage{amsthm}

\newtheorem{theorem}{Theorem}

\newtheorem{corollary}{Corollary}

\newtheorem{lemma}{Lemma}

\title{Sets with no solutions to some symmetric linear equations}
\author{Tomasz Kościuszko}
\begin{document}
\begin{abstract}
We expand the class of linear symmetric equations for which large sets with no non-trivial solutions are known. Our idea is based on first finding a small set with no solutions and then enlarging it to arbitrary size using a multi-dimensional construction, crucially assuming the equation in primitive. We start by presenting the technique on some new equations. Then we use it to show that a symmetric equation with randomly chosen coefficients has a near-optimal set with no non-trivial solutions. We also show a construction for a wide class of symmetric equation in 6 variables. In the final section we present a couple of remarks on non-symmetric equations.
\end{abstract}
\maketitle
\section{Introduction}
Ruzsa in his celebrated paper~\cite{ruzsa} laid the groundwork for analyzing linear equations in integers. Ruzsa suspected that if $A\subseteq [N]=\{1,2,\cdots, N\}$ is the largest set, which does not contain non-trivial solutions to an equation of genus $g$ (we define genus in the next section), then its size is roughly $N^{1/g}$. He showed the conjectured upper bound $O(N^{1/g})$, however, the question remains wide open. Interestingly, as noted in~\cite{gowers}(Proposition 4.14), a positive resolution to the $(9,5)$-free problem in the theory of hypergraphs would give a negative answer to Ruzsa's question.

In this paper we make progress on constructing large sets without solutions to equations of the form
\begin{equation}a_1x_1 + a_2x_2 + \cdots a_k x_k = a_1x_1' + a_2x_2' + \cdots a_k x_k',\tag{1}\end{equation}
where $a_i$ are integers and $x_i,x_i'$ come from a subset $A$ of $[N]$. This is a symmetric linear equation and thus its genus is equal to $k$.
Ruzsa, for any such equation with $k=2$ constructed a set $|A|=\Omega(N^{1/2})$ without non-trivial solutions. In fact, he gives a construction to a wider class of equations of genus 2, but it does not seem to carry-over to equations longer than $4$ variables. He also suggested a greedy algorithm which for any equation of length $m$ gives a set with no non-trivial solutions of size $|A|=\Omega(N^{1/(m-1)})$. In this paper we focus on the lower-bound for symmetric equations in the cases where we can improve upon the bound given by the greedy algorithm. In Section 2 we introduce our technique. It gives near-optimal lower-bounds for some specific equations (Corollary~\ref{cor:two_var}, Corollary~\ref{cor:coprime}) but also for almost any other equation as expressed by Theorem~\ref{thm:random}, which is proven in Section 3.
\begin{theorem}\label{thm:random}
    Let $k\geq 2$ and $\epsilon>0$, then there exists $C_{\epsilon}$, such that for all $C\geq C_{\epsilon}$, there exist at least $(1-\epsilon)C^k$ such $k$-tuples $(a_1, a_2, \cdots, a_k)\in [1, C]^k$ for which the equation
    $$a_1x_1+a_2x_2+\cdots+a_kx_k=a_1x_1'+a_2x_2'+\cdots+a_kx_k'$$
    has no non-trivial solutions in the set $\{1,2,\cdots, \lfloor C^{1/k-\epsilon}\rfloor\}$. Consequently, within the set of $N$ first integers there exists a subset $A$ of size $\Omega(N^{1/k-\epsilon})$, which contains no non-trivial solution to this equation.
\end{theorem}
Subsequently we focus on the obstructions which cause our technique to fail in some of the remaining equations. It turns out that "small" linear combinations of the numbers $a_1,a_2,\cdots, a_k$ which equal 0 create problems if they exist. We manage to partially resolve this difficulty in Section 4, which results in the following theorem.
\begin{theorem}\label{thm:three}
Let $a,b,c$ be positive integers such that $a \leq b \leq c \leq b^{1.01}$ with $b$ large enough, $b,c$ coprime and $c\neq a+b$. For any $N$, there exists a set $A\subseteq [N]$ with $|A| = \Omega(N^{1/4.77})$ such that $A$ has no non-trivial solutions to the equation $ax+by+cz=ax'+by'+cz'$.
\end{theorem}
In the final section we give a couple of remarks on non-symmetric equations. For instance if $d$ is large enough and $A$ is the largest subset of $[N]$ with no non-trivial solutions to the equation
$$x_1+x_2+dx_3+dx_4=2y_1+2dy_2,$$
then 
$$N^{1/2-0.01} \ll |A|\ll e^{-c(\log N)^{1/9}} N^{1/2}.$$
\section{Basic approach and examples}
Let us consider a general linear equation of the form $a_1x_1+a_2x_2+\cdots a_kx_k=0$ where $a_1+a_2+\cdots+a_k=0$. Its genus $g$ is the largest number for which there exist $g$ disjoint sets $\mathcal{Q}_i$ such that $\mathcal{Q}_1\cup \mathcal{Q}_2\cdots\cup \mathcal{Q}_g = [k]$ and $\sum_{j\in \mathcal{Q}_i}a_j = 0$ for every $Q_i$.

A trivial solution~\cite{ruzsa} is such that for all $\alpha$, whenever $\mathcal{T}$ are all indices $i\in [k]$ for which $x_i=\alpha$ we have $\sum_{i\in \mathcal{T}} a_i =0$. 
We call an equation primitive if there exists a partition $$\mathcal{T}_1\cup\mathcal{T}_2\cup\cdots\cup\mathcal{T}_l = [k]$$
such that for every $T_j$ we have $\sum_{i\in T_j}a_i = 0$ and whenever for some $\mathcal{T}\subseteq [k]$ there is $\sum_{i\in T}a_i=0$ it is necessarily a union of some of the sets $T_j$ from the partition.

Let us now focus on symmetric equations (1) where $a_1,a_2,\cdots, a_k$ are coefficients on each side. Notice that such equation is primitive if and only if $a_1, a_2\cdots a_k$ form a dissociated set, which means that sums of all subsets are distinct. In that case a trivial solution means that $x_i=x_i'$ for all $i$.

The first lemma serves as a basic tool to almost all considerations in this paper. It allows us to expand a construction from a fixed, finite set to the set of all integers.

\begin{lemma}\label{lem:base}
    Consider a primitive symmetric linear equation
    $$a_1x_1 + a_2x_2 + \cdots a_k x_k = a_1x_1' + a_2x_2' + \cdots a_k x_k'.$$
    Suppose that there exists a number $L>1$ and a set $A_L\subseteq \{0,1,\cdots, L-1\}$ such that the equation has no non-trivial solutions in the set $A_L$. Suppose that $(|a_1|+|a_2|+\cdots+|a_k|) \max_{x\in A_L} < L$. Then for any integer $N$ there exists a set $A\subseteq [N]$ such that the equation has no non-trivial solutions in $A$ and with $|A|=\Omega(N^r)$, where $r=\frac{\log |A_L|}{\log L}$. Moreover, if additionally $A_L=\{0,1,\cdots, m\}$ for some integer $m$, then $|A|\geq N^r$.
\end{lemma}
\begin{proof}
    Without loss of generality we can assume that $a_1,a_2,\cdots, a_k>0$. If $a_i$ was negative we could swap $x_i$ and $x_i'$ while flipping the sign of $a_i$. If $a_i=0$ we could remove such term as it has no effect on the statement. 
    We represent the numbers from 1 to $N$ in the digit base $L$. Then set the set $A$ to be the numbers with all digits in the set $A_L$. Clearly if $x_1,\cdots, x_k\in A_L$, then $a_1x_1 + a_2x_2 + \cdots a_k x_k < L$. Thus there is no carry-over in base $L$. Suppose that $y_1,y_2,\cdots y_k, y_1',y_2',\cdots y_k'\in A^{2k}$ form a non-trivial solution. That means $y_i\neq y_i'$ for some $i$. Choose any digit in the base $L$ system which is different for $y_i$ and $y_i'$. If we restrict all of the $y_1,y_2,\cdots y_k, y_1',y_2',\cdots y_k'$ to this digit we also have a solution since there is no carry-over. Because the equation is primitive this is a non-trivial solution because of $y_i$ and $y_i'$. Thus any solution in $A$ must be trivial because all solutions in $A_L$ are trivial.\\
    Let us calculate the size of $A$. Suppose first that that $N=L^d$, then since $|A_L|=L^r$ the size of $A$ can be calculated as
    $$|A|=|A_L|^d = L^{rd} = N^r.$$
    If $N$ is a non-integer power of $L$ we can choose the first digit always to be $0$. Then $|A|\geq N^r / L = \Omega(N^r).$\\
    In the second part, we assume that $A_L$ are the $m+1$ smallest digits in the base $L$. We proceed by induction on the number of digits of $N$ in the base $L$. If $N$ has zero digits then $N=0$ and the statement is vacuously true. Now suppose $N$ has $d+1$ digits for $d\geq 0$. Then $N = q L^d + M$, where $q<L$ and $M$ has at most $d$ digits.
    If $q>m$ then
    $$ |A| = (m+1) L^{rd} =  L^{r(d+1)} \geq  N^r$$
    If $q\leq m$ then using the previous step of induction, we have
    
    $$|A| \geq q L^{rd} + M^r\geq (q^{1/r}L^d+M)^r \geq N^r,$$
    by noting $0<r \leq 1$ and thus $f(x)=x^r$ is concave.
\end{proof}
Let us remark here, that such multi-dimensional constructions are in some sense universal for primitive equations. Let $s$ be the sum of absolute values of coefficients. Suppose, that by some unknown construction, we can for all $N$ choose a set $A\subseteq[N]$ with $|A|\geq c N^q$. For a fixed $\epsilon >0$ pick large $L$, so that $ L \geq (2s/c)^{1/\epsilon}$ and fix $A\subseteq [L]$. Let $l = \lfloor L/s \rfloor$. Choose $t$ so that 
$$|A\cap (t+[l])| \geq \lfloor c L^q/s \rfloor \geq \frac{c L^q}{2s}.$$
Our equation is translation-invariant, therefore $A_L = (A\cap (t+[l])) - t$ also does not contain any solutions. Moreover all elements are at most $l$, so we can apply Lemma~\ref{lem:base} to obtain for any $N$ a set of size at least $\Omega(N^r)$, where 
$$r=\frac{q\log L - \log (2s/c)}{\log L}\geq q - \epsilon.$$
Therefore we can approximate any construction by the multi-dimensional construction form Lemma~\ref{lem:base}. The fact that we are unable to find any reasonably big set without solution to the equation $2x+2y=3w+z$ with a computer for small $N$, might therefore suggest that the correct bound is actually closer to $O(N^{\rho})$ for some $\rho <1$. This is of course just speculation as the constant $c$ could be very small.

Here is a well-known example, which could be a first application of Lemma \ref{lem:base}. To our knowledge this is the only equation with optimal lower-bound in primitive equation of more than 4 variables.
\begin{corollary}\label{cor:powers}
    Let $k,m\geq 2$ and $N\geq 1$, there exists a set $A\subseteq [N]$ with $|A|\geq N^{1/k}$, such that it contains no non-trivial solutions to the equation
    $$x_1+mx_2 + \cdots + m^{k-1}x_k=x_1'+mx_2'+ \cdots + m^{k-1}x_k'.$$
\end{corollary}
\begin{proof}
    Let $L=m^k$ and $A_L=\{0,1,\cdots, m-1\}$. By the proposition we obtain $|A|\geq N^r$, where $r=\frac{\log |A_L|}{\log L} = \frac{1}{k}$.
\end{proof}
Our next step is generalizing the argument for $k=2$ to any symmetric equation. Ruzsa already showed in~\cite{ruzsa} how to deal with equations with 4 variables in an optimal way, however in the proof of a later theorem we will need the result below rather than Ruzsa's construction. That is because Ruzsa's construction works for $N > (a+b)^3$ and we will consider smaller $N$.
\begin{corollary}\label{cor:two_var}
    Let $a,b$ be coprime integers with $|a|<|b|$ and let $N>0$, there exists a set $A\subseteq [N]$ with $|A|\geq N^r$, where $$r = \frac{\log |b|}{\log (1+(|a|+|b|)(|b|-1))}\geq \frac{1}{2} - \frac{1}{\log |b|},$$ such that it contains no non-trivial solutions to the equation
    $$ax_1+bx_2=ax_1'+bx_2'.$$
\end{corollary}
\begin{proof}
    Without loss of generality assume that $0<a<b.$
    Let $L=(a+b)(b-1)+1$ and $A_L=\{0,1,\cdots, b-1\}$. Notice that if
    $$ax_1+bx_2=ax_1'+bx_2',$$
    then $a(x_1-x_1')=b(x_2-x_2'),$ so $b|(x_1-x_1')$, therefore $x_1=x_1'$. We get $b(x_2-x_2')=0$ and so $x_2=x_2'$ and the solution is tivial.
    By Lemma~\ref{lem:base} we obtain $|A|\geq N^r$, where $r=\frac{\log b}{\log (1+(a+b)(b-1))}$.
    Notice that for large $b$ the ratio $r$ tends to $1/2$ as
    $$r > \frac{\log b}{\log 2 + 2\log b} \geq \frac{1}{2} - \frac{1}{\log b}.$$
\end{proof}
It follows that if $|b|$ is at least $e^{10^3}$ we have $r>0.499$.
\begin{corollary}\label{cor:no_solutions}
    Let $a,b$ be coprime integers with $|a|<|b|$ and let $N>0$, there exists a set $A\subseteq [N]$ with $|A|\geq N^{\log 6/\log 56}$, such that it contains no non-trivial solutions to the equation
    $$ax_1+bx_2=ax_1'+bx_2'.$$
\end{corollary}
\begin{proof}
    Without loss of generality assume that $0<a<b.$
    By the result above we know that there exists $A\geq N^r$, where
    $$r= \frac{\log b}{\log (1+(a+b)(b-1))}\geq \frac{\log b}{\log (1+(2b-1)(b-1))}.$$
    It is not hard to calculate that such function in $b$ has its minimum at $b=6$ if $b$ is a positive integer. We note that $\frac{\log 6}{\log 56}\in (0.445, 0.446)$.
\end{proof}
By using the above method we extend the class of equations given in Corollary~\ref{cor:powers} with a loss of small $\epsilon$.
\begin{corollary}\label{cor:coprime}
    Let $k,a, b\geq 2$ with $a,b$ coprime and $a\leq b^{k-1}$ let $N\geq 1$. There exists a set $A\subseteq [N]$ with $|A|\geq N^{\frac{1}{k}-\epsilon_{b}}$, such that it contains no non-trivial solutions to the equation
    $$ax_1+ bx_2 +\cdots + b^{k-1} x_k=a x_1' + b x_2'+\cdots + b^{k-1 } x_k'.$$
\end{corollary}
\begin{proof}
    Choose $A_L=\{0,1,\cdots, b-1\}$. By the same reasoning as in Corollary \ref{cor:powers} if $x_2,\cdots, x_k\in A_L$ are not all zero, we have
    $$-a(x_1-x_1')=b(x_2-x_2')+b^2(x_3-x_3')+\cdots+b^{k-1}(x_k-x_k')\neq 0.$$
    This number is divisible by $b$, thus if we have a solution then $b|a(x_1-x_1')$ and $x_1=x_1'$, which is a contradiction with $a(x_1-x_1')\neq 0$. By using Lemma \ref{lem:base} we construct a set $|A|\geq N^r$ for $r=\frac{\log b}{\log(b^k-a-b+ab)}\geq \frac{\log b}{\log(b^k+ab)}\geq \frac{\log b}{k\log b+ \log 2}\geq \frac{1}{k}-\frac{1}{\log b}$.
\end{proof}
Here we can mention one similar construction. For a fixed equation by $R(N)$ denote the largest subset of $[N]$ without solutions to the equation with distinct variables and by $r(N)$ the largest such subset with variables not necessarily distinct. Surprisingly, Ruzsa~\cite{ruzsa} showed that there exists a symmetric equation of any length for which $R(N)\geq N^{1/2-\epsilon}$, while we know that $r(N)\ll N^{1/k}$ for a symmetric equation of length $k$. On the other hand, Bukh~\cite{bukh} showed that for any symmetric equation $R(N)\ll N^{1/2-\epsilon}$ where $\epsilon$ depends on the equation. It is also known~\cite{graph} that any for any invariant equation of genus at least 2 we have $R(N)\ll N^{1/2}$. The question of what happens to $R(N)$ remains open for most equations of genus at least $3$. We can contribute to this problem by building sets without solutions to some equations in distinct variables. For $m\geq 2$ we consider the equation 
$$mx_1 + (2m-2)x_2 + (3m-3)x_3 = mx_1' + (2m-2)x_2' + (3m-3)x_3'.$$
It is not hard to see that if $x_1,x_2,x_3,x_1',x_2',x_3'\in \{0,1,\cdots,m-2\}$ satisfy this equation, we must have $x_1=x_1'$. Using Lemma~\ref{lem:base} we obtain a subset of $[N]$ of size at least $N^{1/2-\epsilon_m}$. It is also possible to construct sets of size $N^{\frac{1}{3}+\epsilon}$ for equations in which coefficients on one side are coprime, for example $43x_1 + 69x_2 + 70x_3 = 43x_1' + 69x_2' + 70x_3'$ (here we get size $\Omega(N^{1/2.96})$ - checked by a computer). This would perhaps support the suspicion that for most equations the behaviour of $r(N)$ and $R(N)$ are very different.

In Corollary \ref{cor:powers} and Corollary \ref{cor:coprime} coefficients grow rapidly as they are powers of an integer. Notice that more generally, if we have an equation
$$a_1x_1+ a_2x_2 +\cdots + a_{k} x_k=a_1 x_1' + a_2 x_2'+\cdots + a_{k} x_k'$$
where $s|a_i|\leq |a_{i+1}|$ for some integer $s$ and all $i < k$, we can choose $A_L=\{0,1,\cdots s-1\}$ as in the proof of Corollary \ref{cor:powers}. Thus we obtain a set $A\subseteq [N]$ for every $N$ for which $|A|\geq N^r$ and $r=\frac{\log s}{\log s + \log (a_1+a_2+\cdots a_k)}$. 
Suppose that $k=3$, $a_1=2$, $a_2=17$, $a_3=167$ (primes). Here we can choose $s=8$ and immediately $r= \frac{\log 8}{\log 8 + \log (2+17+167)}\geq \frac{1}{3.52}$, which is by far better than the greedy construction $N^{1/5}$ and not much worse than the upper-bound $N^{1/3}$. It seems reasonable to ask what happens when the coefficients $a_1 \leq a_2\leq \cdots\leq a_k$ are comparable, for example when $|a_k|\leq |a_1|^{1.01}$. We consider this question in the next two sections.
\section{Average Case}
In the result below we show that almost all (in the probabilistic sense) symmetric equations there exists a lower-bound construction optimal up to small $\epsilon$.
\begin{lemma}
Let $k\geq 2$ and $\epsilon>0$, then for all integers $C\geq C_{\epsilon}$ the following holds. Let $a_1, a_2, \cdots, a_k\in [1, C]$. There are at least $(1-\epsilon)C^k$ such $k$-tuples where the function
$$(i_1,i_2,\cdots, i_k)\rightarrow i_1a_1 + i_2a_2 +\cdots + i_ka_k$$
is injective for all $1\leq i_1,i_2,\cdots,i_k\leq C^{1/k-\epsilon}$. Moreover, $C_{\epsilon}$ can be taken to be equal to $(\epsilon^{-1} 2^k)^{1/(\epsilon k)}$.
\end{lemma}
\begin{proof}
    Let us count the number of tuples for which the function is not injective. Suppose that for some $a_1,a_2,\cdots a_k$ we can find numbers $i_1,i_2,\cdots i_k, i_1', i_2', \cdots i_k'$ such that
    $$a_1(i_1-i_1')+\cdots +a_k(i_k-i_k') = 0.$$ Since from this equation we can determine $a_k$ based on all other variables, $a_k$ can take at most $C^{k-1}\cdot (2\cdot C^{1/k-\epsilon})^k$ many values (at most one value for every assignment of $a_1,\cdots a_{k-1}$ and $i_1-i_1',\cdots, i_k-i_k'$). This value can be estimated as
    $$2^k C^{k-\epsilon k}\leq \epsilon C^k,$$
    which gives us the right number of tuples.
\end{proof}
The proof of Theorem~\ref{thm:random} is now very quick.
\begin{proof}
    We will take
    $$C_{\epsilon} = \Bigl(\frac{k2^k}{\epsilon}\Bigr)^{1/(\epsilon k)},$$
    unless $\epsilon > \frac{1}{k},$ then we take $C_{\epsilon} = C_{1/k}$. Thus we can assume that $\epsilon \leq \frac{1}{k}$.
    
    The first part of the statement is a quick consequence of the previous lemma. Notice that we can use it as
    $$C_{\epsilon}\geq (\epsilon^{-1} 2^k)^{1/(\epsilon k)}.$$
    We also see that any equation coming from obtained tuples is primitive, as otherwise the function would not be injective even with coefficients $\{0,1\}$. For the second part, we use Lemma~\ref{lem:base} with base $$L=kC\lfloor C^{1/(k-1)-\epsilon}\rfloor\geq (a_1+a_2+\cdots+a_k)\lfloor C^{1/(k-1)-\epsilon}\rfloor$$
    and the set $A_L = \{0,1,\cdots, \lfloor C^{1/(k-1)-\epsilon}\rfloor\}$.
     As a result we have $|A|=\Omega(N^r)$, where 
     $$r=\frac{\log\lfloor  C^{1/(k-1)-\epsilon} \rfloor}{\log( kC\lfloor C^{1/(k-1)-\epsilon}\rfloor)}\geq
     \frac{(1/(k-1)-\epsilon)\log C - \log 2}{((k/(k-1)-\epsilon)\log C)+\log k}\geq \frac{1}{k} -\epsilon.$$
     We verify the last inequality directly by
     \begin{equation*}
     \begin{split}
     C &\geq (k2^k)^{1/\epsilon}\\
     \log C &\geq \frac{k\log 2 + \log k}{\epsilon}\\
     \Bigl(\frac{k^2}{k-1}-k\Bigr)\epsilon\log C + k\epsilon\log k&\geq k\log 2 + \log k\\
     \Bigl(\frac{k^2}{k-1}-k\epsilon-(k-1)\Bigr)\epsilon\log C + k\epsilon\log k &\geq k\log 2 + \log k\\
     (k/(k-1)-k\epsilon)\log C - k\log 2&\geq (((k/(k-1)-\epsilon)\log C)+\log k)(1-k\epsilon).
     \end{split}   
     \end{equation*}
     which is equivalent to the desired inequality via multiplying by the denominators.
\end{proof}
\section{More general result for 6 variables}
In the above argument the main obstacle is linear dependence between numbers $a$, $b$, $c$. For example, if $2a+b=c$ we cannot pick $L>1$ so that $i,j,k\in [-L,L]\setminus \{0\}$ have $ia+jb+kc \neq 0.$ In some cases we can however pick a subset of $[-L,L]$ for a larger $L$. For example consider $a=10, b=11, c=31$, which satisfy such linear dependency. The set $A_0=\{0,1,4,5\}$ will be a good subset of digits in a base $5\cdot(a+b+c)+1=260$ digit system. By Lemma \ref{lem:base}, for any $N$ we can pick $A\subseteq N$ without solutions to $$10x+11y+31z=10x'+11y'+31z'$$
and with $|A|=\Omega(N^{4.02}).$ Here the set $A_0-A_0$ contains no numbers of the form $x$, $2x$ and thus we managed to avoid the small linear combination $2a+b=c$ and its multiples. It turns out that in general we can easily avoid one such linear dependency as is presented below.
\begin{lemma}\label{lem:one_solution}
    Let $ia + jb + ck=0$ be a linear equation in variables $i,j,k$, such that the only solutions in the set $\{-L,\cdots, L\}$ are of the form $t(i_1, j_1, k_1)$ for some rational number $t$ and a basic solution $(i_1, j_1, k_1)$. Assume that $|i_1| < |j_1|$. Then there exists a set $A \subseteq [L]$, with $|A|\geq L^{0.44}$, such that $A-A$ contains no non-zero solutions to the equation $ia + jb + ck=0$. Moreover, if $|j_1|/GCD(|i_1|, |j_1|) > C = e^{10^3}$, we have $|A|\geq L^{0.499}$.
\end{lemma}
\begin{proof}
    Without loss of generality suppose that $(i_1, j_1, k_1)$ is the smallest solution, i.e. there does not exist an integer $d>1$ for which $d|i_1, d|j_1, d|k_1$. By Corollary \ref{cor:no_solutions} there exists subset $A$ of size $|A|\geq L^{\log 6/\log 56}\geq L^{0.44}$ ($|A|\geq L^{0.499}$ if $|j_1|/GCD(|i_1|, |j_1|) > C = e^{10^3}$) with no solutions to the equation
    $$i_1 x_1 + j_1 x_2 = i_1 x_1' + j_1 x_2'.$$
    Thus for any elements $z_1 = x_1-x_1'\in A-A, z_2=x_2'-x_2\in A-A$, where at least one is not zero, we have
    $$i_1 z_1 \neq j_1 z_2.$$
    If we had a solution to $ia+jb+ck=0$ in $A-A\subseteq \{-L, \cdots, L\}$ it would be of the form $t(i_1, j_1, k_1)$ for some rational number $t$. Thus $t i_1,t j_1\in A-A$. Trivially
    $$j_1 (ti_1) = i_1(tj_1),$$
    which is a contradiction with above inequality considering $z_1=ti_1, z_2=tj_1$.
\end{proof}
We extract part of the proof to the lemma below. The idea here is that there is at most one "small" linear combination of $a,b,c$ we need to avoid. It would be interesting to see if the exponent in Lemma~\ref{lem:small_large} can be improved from $1-\alpha_2$ to something larger.
\begin{lemma} \label{lem:small_large}
    Let $a,b,c$ be positive integers, with $a\leq b\leq c$ and with $b, c$ coprime. Suppose that for a real number $\alpha_2\in (0,1/2)$ there exist $i_1, j_1, k_1$, not all zero, with  $|i_1|,|j_1|,|k_1|\leq b^{\alpha_2}$  such that
    $$i_1a+j_1b+k_1c=0.$$
    If integers $i_2, j_2, k_2$ are not of the form $t i_1, t j_1, t k_1$ for any $t$ and $i_2a + j_2b + k_2 c =0$ then
    $$\max(|i_2|, |j_2|, |k_2|) \geq \frac{1}{2}b^{1-\alpha_2}.$$
\end{lemma}
\begin{proof}
We have $i_1\neq 0$ because otherwise $j_1 b = k_1 c$, which means that $b|k_1$, but $|k_1|<b$ so $k_1=0$. It follows that $i_1=j_1=k_1=0$, which is not the case.
Let $V$ be the subspace of $\mathbb{R}^3$ perpendicular to $(a,b,c).$ Clearly $\dim V = 2$. Notice that $(i_1,j_1,k_1)\in V$ and it is not perpendicular to $(1,0,0)$. Therefore the subspace of $V$ perpendicular to $(1,0,0)$ has dimension 1. On the other hand $$
(0, i_2 j_1 - i_1 j_2, i_2 k_1 - i_1 k_2),(0,c,-b)\in V$$ and both are perpendicular to $(1,0,0)$, thus they must be collinear. Therefore 
$$c(i_2 k_1 - i_1 k_2) = -b(i_2 j_1 - i_1 j_2)$$
and since $b$ and $c$ are coprime, $c|i_2 j_1 - i_1 j_2$ and $b|i_2 k_1 - i_1 k_2$. If both numbers are $0$, then $(i_2,j_2,k_2)=t(i_1,j_1,k_1)$, which is a contradiction. Therefore, at least one of the two holds:
$$|i_2 j_1 - i_1 j_2| \geq c \geq b,$$
$$|i_2 k_1 - i_1 k_2| \geq b.$$
If $|i_2|\geq b^{1-\alpha_2}/2$ the proof is finished, so let us assume the opposite. If the first of the two inequalities is true we have by the triangle inequality
$$|i_1j_2|\geq b - |i_2 j_1|\geq b - |j_1|b^{1-\alpha_2}/2 \geq b/2$$
and therefore
$$|j_2| \geq \frac{b}{2|i_1|}\geq \frac{1}{2}b^{1-\alpha_2}.$$
In case of the other inequality being true the reasoning is exactly the same.
\end{proof}
We will consider the equation 
$$ax+by+cz=ax'+by'+cz',$$
where $0<a\leq b\leq c$. In the next theorem, one key assumption is $c\leq b^{1.01}$. Notice that if $c > b^3$, then a simple construction can be given. We choose $A_L\subseteq [\lfloor c^{2/3}/2\rfloor]$ to not contain solutions to $ax+by=ax'+by'$ (using Corollary~\ref{cor:two_var}). Then if $x,y,x',y'\in A_L$ we have $|ax+by-ax'-by'|<c$ so this set also does not contain solutions to the equation $ax+by+cz=ax'+by'+cz'$. Choosing $L=3c$ we have $\frac{\log |A_L|}{\log L}\geq \frac{\frac{1}{3}(\log c)-o_b(1)}{\log 3 + \log c}\geq \frac{1}{3} - o_b(1)$. We do not know what happens if $c\in (b^{1.01}, b^3)$, however if $b$ and $c$ are coprime we know how to approach the case $c\leq b^{1.01}$. 
Equipped with Lemma~\ref{lem:small_large} we can prove Theorem~\ref{thm:three}.

\begin{proof}
We fix a real number $\alpha\in (0,1/2)$ for which we will optimize later. 
Because $c\neq a+b$ our equation is primitive. 
Consider the equation $$i a + j b + k c = 0.$$ If there are no solutions to this equation with $|i|,|j|,|k| \leq b^{\alpha}$ we consider the set $[b^{\alpha}]\subseteq [(a+b+c) b^{\alpha}]$ as $A_0$. It works for us since all the elements of $A_0-A_0 = [-b^{\alpha}, b^{\alpha}]$ have absolute value at most $b^{\alpha}$, so it is impossible to choose $i$, $j$, $k$ from $A_0-A_0$.
If we then choose the digit base to be $(a+b+c) b^{\alpha}< 4c b^{\alpha}\leq b^{\log_b(c) + \alpha + o_b(1)}$, by Lemma~\ref{lem:base} we can choose $A\subseteq [N]$ of size $\Omega(N^{\frac{\alpha}{\log_b(c) + \alpha + o_b(1)}})$.

Otherwise let $i_1,j_1,k_1$ be one solution with $|i|,|j|,|k| \leq b^{\alpha_2}$ ($\alpha_2$ will be later either $\alpha$ or something smaller). We assume that $GCD(|i_1|, |j_1|, |k_1|)=1$ because if that was not the case we could divide each number by the common divisor. By Lemma~\ref{lem:small_large} if there exist a solution $i_2, j_2, k_2$, which is not a multiple of this one we have $\max(|i_2|, |j_2|, |k_2|)\geq b^{1-\alpha_2}/2$.

Let us restrict our attention to the set $(-b^{1-\alpha_2}/2, b^{1-\alpha_2}/2)$. By the above argument, the only solutions to the equation $ia + jb + kc = 0$ in this set are of the from $t'(i_1, j_1, k_1)$. Since the equation is primitive not all of $|i_1|,|j_1|,|k_1|$ are the same. For example it could be that $|i_1|<|j_1|$. Apply Lemma \ref{lem:one_solution} and deduce that there exists a subset $A_0$ of $[b^{1-\alpha_2}/2]$ of size at least $b^{0.44(1-\alpha_2)}/2$, such that $A_0-A_0$ never contains both $i_1x$ and $j_1x$ for a given integer $x$. Thus we managed to avoid all possible solutions.
The constant $0.44$ does not always satisfy us, so we will improve upon it by considering the following two cases.

Suppose that $\max(|i_1|,|j_1|,|k_1|) > C = e^{10^6}.$
If say, $|i_1|$ is the largest, because $GCD(|i_1|, |j_1|, |k_1|)=1$ we have $|i_1|/GCD(|i_1|,|j_1|)\geq e^{10^3}$ or $|i_1|/GCD(|k_1|,|i_1|)\geq e^{10^3}$. Then take $\alpha_2=\alpha$ and by Lemma \ref{lem:one_solution} we have $|A_0|\geq b^{0.499(1-\alpha)}/2$. The same holds, analogously if $|j_1|$ or $|k_1|$ are the largest.

If on the other hand $|i_1|, |j_1|, |k_1|\leq C$ then for large enough $b$ we have for $\alpha_2=0.1$ that $|i_1|, |j_1|, |k_1|\leq C \leq b^{\alpha_2}$ and so by Lemma~\ref{lem:one_solution} we can construct $A_0$ with $|A_0|\geq b^{0.44(1-\alpha_2)}/2 \geq b^{0.39}$.

We have constructed $A_0\subseteq [b^{1-\alpha_2}/2]$ to have no non-trivial solutions to the equation $ax+by+cz=ax'+by'+cz'$. Considering digit base $(a+b+c)b^{1-\alpha_2} < 4cb^{1-\alpha_2}$ in Lemma~\ref{lem:base}, for any $N$ we can construct set $A\subseteq [N]$ with $|A|=\Omega(N^{\frac{0.499(1-\alpha)}{\log_b(c) + 1 - \alpha + o_b(1)}})$ in the first case and $|A|=\Omega(N^{\frac{0.39}{\log_b(c) + o_b(1) +0.39}})$ in the second case.

We have shown three constructions of $A$. Depending on the nature of the equation and the constant $\alpha$ we always select one of them to get $A$. We now wish to optimize the size of $A$ by choosing suitable $\alpha$. From the above considerations we know that
$$|A|\geq c \min\Bigl(N^{\frac{\alpha}{\log_b(c)+\alpha+o_b(1)}}, N^{\frac{0.499(1-\alpha)}{\log_b(c)+1-\alpha+o_b(1)}}, N^{\frac{0.39}{\log_b(c)+0.39+o_b(1)}}\Bigr),$$
for some small absolute constant $c$. To find the optimal value of $\alpha$ we solve the corresponding quadratic equation
$$\alpha (1+\log_b(c)-\alpha) = 0.499 (1-\alpha)(\log_b(c)+\alpha).$$
The only solution in the interval $(0,1)$ is
$$\alpha  = \frac{-1 + q - \beta - q \beta +  \sqrt{4 (-1 + q) q \beta + (1 + q (-1 + \beta) + \beta)^2}}{2 (q - 1)},$$
where $q=0.499$ and $\beta=\log_b(c).$ For such $\alpha$ we have $|A|=\Omega \Bigl( N^{\frac{\alpha}{\beta + \alpha + o_b(1)}} \Bigr)$. To give some concrete values, if $c=b$ we have $|A| =\Omega( N^{1/4.74})$ and when $c=b^{1.01}$ we have $|A| = \Omega( N^{1/4.77})$. On the other hand if $c=b^{1.1}$ we get $|A|=\Omega( N^{1/5.03})$, which is worse than the trivial bound.

\end{proof}
Even though our technique gives promising results for a wide class of equations, we are still unable to answer the question of Zhao~\cite{100problems}, who asks for a large set with no non-trivial solutions to the non-primitive equation $x+2y+3z=x'+2y'+3z'$. 
\section{Non-symmetric equations}
The case of non-symmetric equations appears to be more difficult. The result of Ruzsa shows that if an equation has genus $g$ then any set without solutions must have size $O(N^{1/g})$, however not much more is known. Our technique does not seem to work well for the non-symmetric case but there are some remarks we can make about very specific equations.

Prendiville~\cite{prend} showed a bound on subsets of Sidon sets without solutions to invariant linear equations of length at least 5. It was later boosted by a result of the author~\cite{kosc}. We can interpret this result as solving the equation
$$a_1 x_1 + a_2 x_2 + \cdots + a_k x_k + ay = az + aw,$$
where $k\geq 4$ and $a = a_1 + a_2 + \cdots + a_k.$ The bound we obtain for this equation of genus 2 is stronger than $O(N^{1/2})$.

If a set $A$ is not Sidon then there exist numbers $x,y,z,w\in A$ with $x\neq z$ and $x\neq w$ such that $x+y=z+w$. So our equation has a non-trivial solution $x_i=x$ and $y,z,w$. If $A$ is Sidon and $|A|\gg N^{1/2}\exp(-C(\log\log)^{1/7})$, then by Prendiville's result, there exists a non-trivial solution to the equation $a_1x_1 + a_2x_2 + \cdots a_k x_k - az = 0$ in $A$. Taking $y=w$ we have a non-trivial solution to our equation. Thus $|A|\ll N^{1/2}\exp(-C(\log\log)^{1/7})$.\\

Another equation we will consider is
$$x_1+x_2+dx_3+dx_4=2y_1+2dy_2.$$
For any integer $d\geq 1$, we are able to show that if $A$ contains no non-trivial solutions to this equation, then $|A|\ll e^{-c(\log N)^{1/9}} N^{1/2}$. Moreover, we can give an example of such set for large $d$, with $|A|\gg N^{1/2-C/\sqrt{\log d}}$.

Clearly $A$ cannot contain $x,y,w,z$ with $x\neq w$ and $x\neq z$ such that $x+dy=w+dz$ because this would be a non-trivial solution to our equation. We proceed to estimate the size of the set
$$A\dot{+}dA = \{a_1+da_2 : a_1,a_2\in A,a_1\neq a_2\}.$$
It cannot contain any arithmetic progression on length 3 because then we have $x_1,x_2,x_3,x_4,y_1,y_2\in A$ with $$(x_1+dx_2)+(x_3+dx_4)=2(y_1+dy_2)$$ and $x_1+dx_2<y_1+dy_2<x_3+dx_4$. It cannot be a trivial solution because if $x_1=x_3=y_1$ and $x_2=x_4=y_2$ the inequality is violated. On the other hand when $d=1$ then $x_1=x_2=y_1$ and $x_3=x_4=y_2$ is impossible as $x_1+dx_2\in A\dot{+}dA$ and this is the only representation. Therefore by the Kelley-Meka(\cite{km},~\cite{km_impr}) bound
$$|A|^2-|A| = |A\dot{+}dA|\leq e^{-c(\log N)^{1/9}} N.$$
As for the construction, we let $m = \lfloor (d-1)/2 \rfloor$. By Behrend's construction, we can find $A_L \subseteq\{0,1,\cdots, m\}$ such that it contains no non-trivial arithmetic progressions of length 3 with $|A_L|\gg m\exp(-C(\log m)^{1/2})m$. Let $L=(4m+3)m+1$. There are no solutions to our equations in $A_L$ because we must have $x_1+x_2=2y_1$ and $x_3+x_4=2y_2$ and both of them must be trivial progressions by the construction of $A_L$. By Lemma~\ref{lem:base} we can obtain $|A|\gg N^r$, where $$r=\frac{\log |A_L|}{\log L}\geq \frac{\log m - C(\log m)^{1/2}}{4 + 2\log m}\geq \frac{1}{2} - \frac{C}{\log^{1/2} m}.$$
Notice that in the argument any convex equation could be taken instead of the equation for 3-term arithmetic progressions. We could also split each variable into more that two terms by introducing higher powers of $d$, thus giving upper and lower bounds for a larger class of equations.

We conclude with a simple observation based on Lemma~\ref{lem:base}. If for a symmetric equation
$$a_1x_1 + a_2x_2+\cdots a_kx_k = a_1x_1' + a_2x_2'+\cdots a_kx_k',$$ where $a_1,a_2,\cdots, a_k>0$
there is a set $A_L\subseteq [L]$ with $(a_1+a_2+\cdots+a_k)\max_{x\in A_L} <L$, then for any integers $i_1,i_2,\cdots, i_k, j_1,j_2, \cdots, j_k$ the set $A_L$ does not contain any solutions to the equation
$$(i_1 L +a_1)x_1 + \cdots +(i_k L + a_k)x_k = (j_1L + a_1)x_1' +\cdots +(j_k L + a_k)x_k'.$$
This is of course because any solution would be a solution modulo $L$, which does not exist by our construction. If now $s=i_1+i_2+\cdots+ i_k+ j_1+j_2+ \cdots+ j_k$ we are able to construct $|A|=\Omega(N^r)$, where $r=\frac{\log|A_L|}{\log L + \log s}$.

\textsc{Faculty of Mathematics and Computer Science, Adam Mickiewicz University, Umultowska 87, 61-614 Poznan, Poland }

\textit{Email address:} tomasz.kosciuszko@amu.edu.pl
\end{document}